\documentclass[leqno,11pt]{amsart}
\usepackage{amsmath,amssymb}
\usepackage{color}
\usepackage{comment}
\usepackage{todonotes}

\allowdisplaybreaks

\theoremstyle{plain}
\newtheorem{theorem}{Theorem}[section]

\newtheorem*{theorem-A}{Theorem A}
\newtheorem*{theorem-B}{Theorem B}
\newtheorem*{theorem-C}{Theorem C}
\newtheorem*{theorem-D}{Theorem D}
\newtheorem*{theorem-E}{Theorem E}
\newtheorem*{theorem-F}{Theorem F}
\theoremstyle{definition}

\def\rn{\mathbb R\sp n}

\def\R{\mathbb R}
\def\N{\mathbb N}
\def\S{\mathbb S}

\def\limsup{\operatornamewithlimits{lim\,sup}}

\def\o{\Omega}

\def\ass{b}
\def\Ass{B}

\newtoks\by
\newtoks\paper
\newtoks\book
\newtoks\jour
\newtoks\yr
\newtoks\pages
\newtoks\vol
\newtoks\publ
\newtoks\eds
\newtoks\proc
\newtoks\no
\def\ota{{\hbox{???}}}
\def\cLear{\by=\ota\paper=\ota\book=\ota\jour=\ota\yr=\ota
\pages=\ota\vol=\ota\publ=\ota}
\def\endpaper{\the\by, \textit{\the\paper},
{\the\jour} \textbf{\the\vol} (\the\yr), \the\pages.\cLear}
\def\endbook{\the\by, \textit{\the\book}, \the\publ.\cLear}
\def\endprep{\the\by, \textit{\the\paper}, \the\jour.\cLear}
\def\endproc{\the\by, \textit{\the\paper}, \the\publ, \the\pages.\cLear}
\def\name#1#2{#1 #2}
\def\et{ and }

\numberwithin{equation}{section}

\newcommand{\norm}[1]{{\left\vert\kern-0.25ex\left\vert\kern-0.25ex\left\vert #1
    \right\vert\kern-0.25ex\right\vert\kern-0.25ex\right\vert}}

\begin{document}

\title[Limit as $s\to 0^+$ of fractional Orlicz-Sobolev spaces]{On the limit as $s\to 0^+$ of fractional Orlicz-Sobolev spaces}

\author {Angela Alberico, Andrea Cianchi, Lubo\v s Pick and Lenka Slav\'ikov\'a}

\address{Angela Alberico, Istituto per le Applicazioni del Calcolo ``M. Picone''\\
Consiglio Nazionale delle Ricerche \\
Via Pietro Castellino 111\\
80131 Napoli\\
Italy} \email{a.alberico@iac.cnr.it}

\address{Andrea Cianchi, Dipartimento di Matematica e Informatica \lq\lq U. Dini"\\
Universit\`a di Firenze\\
Viale Morgagni 67/a\\
50134 Firenze\\
Italy} \email{andrea.cianchi@unifi.it}

\address{Lubo\v s Pick, Department of Mathematical Analysis\\
Faculty of Mathematics and Physics\\
Charles University\\
Sokolovsk\'a~83\\
186~75 Praha~8\\
Czech Republic} \email{pick@karlin.mff.cuni.cz}

\address{Lenka Slav\'ikov\'a, Mathematical Institute, University of Bonn, Endenicher Allee 60,
53115 Bonn, Germany}
\email{slavikova@karlin.mff.cuni.cz}
\urladdr{}


\begin{abstract}
An extended version of the Maz'ya-Shaposhnikova theorem on the limit as $s\to 0^+$ of the Gagliardo-Slobodeckij fractional seminorm is established in the  Orlicz space setting.  Our result holds in fractional Orlicz-Sobolev spaces associated with   Young functions satisfying the $\Delta_2$-condition, and, as shown by counterexamples,  it may fail if this condition is dropped.
\end{abstract}


\maketitle

\section{Introduction and main results}


Pivotal instances of spaces  of functions endowed with a non-integer order of smoothness are the Besov spaces, defined in terms of norms of differences, the Triebel--Lizorkin spaces, whose notion relies upon the Fourier transform, the Bessel potential spaces, based on representation formulas via potential operators, and the Gagliardo-Slobodeckij spaces, defined in terms of fractional difference quotients.
Relations among these families of spaces
are known -- see e.g.  \cite[Remark~2.1.1]{ST} for a survey of results with this regard. It is also well known that, with the exception of the Bessel potential spaces, they do not agree, in general, with the classical integer-order Sobolev spaces when the order of smoothness is formally set to an integer.
\par
In particular, this drawback   affects the Gagliardo-Slobodeckij spaces $W^{s,p}(\rn)$, which are defined, for $n \in \mathbb N$, $s \in (0,1)$ and $p \in [1, \infty)$, via a seminorm depending on an integral over $\rn \times \rn$ of an $s$-th-order difference quotient. However, some twenty years ago it was discovered that a suitably normalized    Gagliardo-Slobodeckij seminorm in $W^{s,p}(\rn)$ recovers, in the limit as $s \to 1^-$ or $s\to 0^+$, its integer-order counterpart.
\par The result was first established at the endpoint $1^-$ by
Bourgain, Brezis and Mironescu in \cite{BBM1, BBM2}. In those papers it is shown that the seminorm in $W^{s,p}(\rn)$ of a function $u$, times $(1-s)^{1/p}$, approaches the $L^p$ norm of $\nabla u$ as $s \to 1^-$
 (up to a multiplicative constant depending only on $n$).
\par
The problem concerning  the opposite endpoint $0^+$   was solved by Maz'ya and Shaposhnikova. In  \cite{MS} they proved that
\begin{equation}\label{feb1}
\lim_{s\to 0^+}  s \int_{\rn} \int_{\rn}\left(\frac{|u(x)-u(y)|}{|x-y|^s}\right)^p\frac{dx\,dy}{|x-y|^n} = \frac{2\, \omega_n}{np} \int_{\rn} |u(x)|^p\; dx
\end{equation}
for every function $u$ decaying to $0$ near infinity and making the double integral finite for some $s\in (0,1)$.
Here, $\omega_n$ denotes the Lebesgue measure of the unit ball in $\rn$.
%
\par The present paper   deals with a version of property \eqref{feb1}   in the broader framework of  fractional Orlicz-Sobolev spaces. These spaces extend the spaces $W^{s,p}(\rn)$ in that the role of the power function $t^p$ is played by a more general Young function $A: [0, \infty) \to [0, \infty)$, namely a convex function vanishing at $0$. Specifically, we address the problem of the existence of
\begin{equation}\label{feb2}
\lim_{s\to 0^+}  s\int_{\R^n} \int_{\rn} A\left(\frac{|u(x)-u(y)|}{|x-y|^s}\right)\frac{dx\,dy}{|x-y|^{n}}\,,
\end{equation}
 and of its value in the affirmative case.  The ambient space for $u$ is  $\bigcup _{s\in (0,1)}V^{s,A}_d(\rn)$, where $V^{s,A}_d(\rn)$ denotes the space of those measurable functions $u$ in $\rn$ which render the double  integral in \eqref{feb2} finite, and decay to $0$ near infinity, in the sense that
$$|\{x \in \rn: |u(x)|>t\}| <\infty \qquad \text{for every $t>0$.}$$
Here, $|E|$ stands for the Lebesgue measure of a set $E\subset \rn$.
\par A partial result in this connection is contained in the recent contribution \cite{CMSV}, where    bounds for the $\liminf_{s\to 0^+}$ and $\limsup_{s\to 0^+}$ of the expression under the limit in \eqref{feb2} are given for Young functions $A$ satisfying the $\Delta_2$-condition. Recall that this condition amounts to requiring that there exists a constant $c$ such that
\begin{equation}\label{delta2}
A(2t) \leq c A(t) \qquad \text{for $t \geq 0$.}
\end{equation}
\par
Our results provide a full answer to the relevant problem. We prove that, under the $\Delta_2$-condition on $A$, the limit in \eqref{feb2} does exist, and equals the integral of a function of $|u|$ over $\rn$. Moreover,  we show that the result can fail if the
$\Delta_2$-condition
is dropped.
Interestingly, the function of $|u|$ appearing in the   integral obtained in the limit is not $A$, but rather the Young function $\overline A$
associated with $A$ by the formula
\begin{equation}\label{Abar}
\overline A(t) = \int _0^t\frac {A(\tau)}\tau \, d\tau \qquad \text{for $t\geq0$\,.}
\end{equation}
Notice that  $A$ and $\overline A$ are equivalent as Young functions, since
$A(t/2) \leq \overline A(t) \leq  A(t)$ for $t \geq 0$,
owing to the monotonicity of $A(t)$ and  $A(t)/t$.

\begin{theorem}\label{T:lim0}
Let $n \in \mathbb N$ and let $A$ be a Young function satisfying the $\Delta_2$-condition. Assume that $u \in \bigcup_{s\in (0,1)} V^{s,A}_d(\rn)$. Then
\begin{equation}\label{jan1}
\lim_{s\to 0^+}  s\int_{\R^n} \int_{\rn} A\left(\frac{|u(x)-u(y)|}{|x-y|^s}\right)\frac{dx\,dy}{|x-y|^{n}} = \frac{2\, \omega_n}n \int_{\rn} \overline A(|u(x)|)\; dx\,.
\end{equation}
\end{theorem}
Plainly, equation \eqref{jan1} recovers \eqref{feb1} when $A(t)=t^p$ for some $p \geq 1$, inasmuch as $\overline A(t)= t^p/p$ in this case.

\smallskip
\par
The indispensability  of the $\Delta_2$-condition for the function $A$ is demonstrated via the next result.
\begin{theorem}\label{counterex}
Let $n \in \mathbb N$. There exist   Young functions $A$, which do not satisfy the $\Delta_2$-condition, and corresponding  functions $u:\rn \to\R$ such that $u \in \ V^{s,A}_d(\rn)$  for every $s\in(0,1)$,
\begin{equation}\label{A2}
\int_{\rn} \overline A(|u(x)|)\;dx \leq \int_{\rn} A(|u(x)|)\;dx <\infty\,,
\end{equation}
but
\begin{equation}\label{A1}
\lim_{s\to 0^+}\,s\int_{\rn} \int_{\rn} A \left( \frac{|u(x) - u(y)|}{|x-y|^s}\right)\; \frac{dx\, dy}{|x-y|^n} =\infty\,.
\end{equation}
\end{theorem}
Incidentally, let us mention that an analogue of the Bourgain-Brezis-Mironescu theorem on the limit as $s\to 1^-$ for   fractional Orlicz-Sobolev spaces built upon Young functions satisfying the $\Delta_2$-condition can be found in \cite{BonderSalort}. Such a condition is removed in a version of this result offered in \cite{ACPS_limit1}. Further properties and applications of fractional Orlicz-Sobolev spaces are the subject of   \cite{ACPS_frac, BO, NBS, Sa}.

\section{Proof of Theorem \ref{T:lim0}}

Our approach to Theorem \ref{T:lim0} is related to that of \cite{MS}, yet calls into play specific Orlicz space results and techniques. In particular,
it makes critical use of a Hardy type inequality for functions in $V^{s,A}_d(\rn)$, with $s\in (0,1)$, recently established in \cite[Theorem 5.1]{ACPS_frac}. This inequality  tells us what follows.
\\ Given a Young function $A$, denote by $a: [0, \infty) \to [0, \infty)$   the  left-continuous non-decreasing function such that
\begin{equation*}
    A(t)=\int_0^{t}a(\tau) d\tau \qquad\text{for $t\geq 0$\,.}
\end{equation*}
Assume that
\begin{equation}\label{E:0'}
	\int^{\infty}\left(\frac{t}{A(t)}\right)^{\frac{s}{n-s}}\;dt = \infty
\end{equation}
and
\begin{equation}\label{E:0''}
	\int_{0}\left(\frac{t}{A(t)}\right)^{\frac{s}{n-s}}\;dt < \infty\,.
\end{equation}
Call $\Ass$ the Young function defined by
\begin{equation*}
    \Ass(t)=\int_0^{t}\ass(\tau) d\tau \qquad\text{for $t\geq 0$,}
\end{equation*}
where the (generalized) left-continuous inverse of the function $\ass$ obeys
\begin{equation}\label{E:2}
	\ass^{\,-1}(r) = \left(\int_{a^{-1}(r)}^{\infty}
		\left(\int_0^t\left(\frac{1}{a(\varrho)}\right)^{\frac{s}{n-s}}\,d\varrho\right)^{-\frac{n}{s}}\frac{dt}{a(t)^{\frac{n}{n-s}}}
				\right)^{\frac{s}{s-n}}
					\qquad\text{for $r\ge0$}\,.
\end{equation}
Then, there exists a constant $C=C(n, s)$ such that
\begin{equation}\label{jan5}
	\int_{\R^n}{\Ass}\left(\frac{|u(x)|}{|x|^s}\right)\;dx
		\le  (1-s) \int_{\R^n} \int_{\R^n}
			A\left(C\frac{|u(x)-u(y)|}{|x-y|^s}\right)\,\frac{dx\, dy}{|x-y|^n}
\end{equation}
for every   function $u \in  V_d^{s,A} (\rn)$. Moreover, the constant $C$ is uniformly bounded in $s$ if $s$ is bounded away from $1$.

\begin{proof}[Proof of Theorem \ref{T:lim0}]
Inasmuch as $A$ satisfies the  $\Delta _2$-condition,    its upper Matuszewska-Orlicz  index  $I(A)$, defined as
\begin{equation}\label{index}I(A) = \lim_{\lambda \to \infty} \frac{\log \Big(\sup _{t>0}\frac{A(\lambda t)}{A(t)}\Big)}{\log \lambda}\,,
\end{equation}
is finite.
A standard (and easily verified) consequence of this fact is that
there exists a positive constant $C=C(A)$ such that
\begin{equation}\label{E:31}
	A(\lambda t)\le C \lambda^{I(A)+1}A(t)\qquad\hbox{for $t\ge0$ and $\lambda\ge1$.}
\end{equation}
%
%
Thereby, there exists $s_0\in (0, 1)$ such that conditions  \eqref{E:0'} and \eqref{E:0''} are  fulfilled if $s\in (0, s_0)$. Hence, inequality \eqref{jan5} holds for $s\in (0, s_0)$.
\\ On the other hand,
$I(A) < \frac ns$ provided that $s<\frac{n}{I(A)}$. Hence,   \cite[Propositions 5.1 and 5.2]{cianchi_Ibero} ensure that the function $\Ass$ is equivalent to $A$ if $s<\frac{n}{I(A)}$. Namely, there exist constants $c_2>c_1>0$ such that $A(c_1t) \leq \Ass (t) \leq A(c_2 t)$ for $t \geq 0$.
Set $s_1 = \min \{s_0, \frac{n}{I(A)}\}$.  As a consequence of
 inequality \eqref{jan5}, of the equivalence of $A$ and $B$, of the $\Delta_2$-condition for $A$, and  of the inequality $\overline{A}\leq A$, if
  $u \in  V_d^{s,A} (\rn)$ for some $s \in (0,  s_1)$, then
\begin{equation}\label{jan6}
\int_{\R^n}{\overline A}\left(\frac{ |u(x)|}{\lambda|x|^s}\right)\;dx \leq \int_{\R^n}{A}\left(\frac{ |u(x)|}{\lambda|x|^s}\right)\;dx  < \infty
\end{equation}
for every $\lambda >0$.\\
We begin by establishing a lower bound for the $\liminf _{s\to 0^+}$ of the expression on the left-hand side of equation \eqref{jan1}. One has that
\begin{align}\label{E:2'}
		\int_{\R^n}
	\int_{\{|x-y|>2|x|\}} A\left(\frac{|u(x)|}{|x-y|^s}\right)\frac{dx\,dy}{|x-y|^{n}}
		& =  \frac{\omega_n}{n}\int_{\rn} \int_{2|x|}^{\infty}A\left(\frac{
	|u(x)|}{r^{s}}\right)\frac{dr}{r}\;dx
		= \frac{\omega_n}{ns} \int_{\rn} \overline{A}\left(\frac{
	|u(x)|}{2^s|x|^{s}}\right)\;dx\,.
\end{align}
Fix $\varepsilon>0$. Owing to the convexity of $A$,
\begin{align}\label{E:3}
		\int_{\R^n}\int_{\{|x-y|>2|x|\}}
		& A\left(\frac{|u(x)|}{|x-y|^s}\right)\frac{dx\,dy}{|x-y|^{n}}
		 \le \frac{1}{1+\varepsilon}\int_{\rn} \int_{\{|x-y|>2|x|\}}A\left((1+\varepsilon)\frac{
				|u(x)-u(y)|}{|x-y|^s}\right)\frac{dx\,dy}{|x-y|^{n}}
			\\ \nonumber
		& \quad + \frac{\varepsilon}{1+\varepsilon} \int_{\rn} \int_{\{|x-y|>2|x|\}}A\left(\frac{1+\varepsilon}{\varepsilon}\frac{
				|u(y)|}{|x-y|^s}\right)\frac{dx\,dy}{|x-y|^{n}} = I_1 + I_2.
\end{align}
Consider the integral $I_2$. If $|x-y| > 2|x|$, then $\frac{2}{3}|y|\le|x-y|\le2|y|$. Therefore,
\begin{align}\label{E:4}
		I_2 & \le \frac{\varepsilon}{1+\varepsilon} \int_{\rn}\int_{\{|x-y|>2|x|\}}A\left(\frac{1+\varepsilon}{\varepsilon}
				\left(\frac{3}{2}\right)^{s}\frac{|u(y)|}{|y|^{s}}\right)\left(\frac{3}{2}\right)^{n}\frac{dy}{|y|^{n}}\;dx
				\\ \nonumber
			& = \frac{\varepsilon}{1+\varepsilon}\left(\frac{3}{2}\right)^{n} \int_{\rn}\frac{1}{|y|^{n}}
				A\left(\frac{1+\varepsilon}{\varepsilon}\left(\frac{3}{2}\right)^{s}\frac{|u(y)|}{|y|^{s}}\right)
				\left(\int_{|x-y|\ge2|x|}\;dx\right)\;dy
				\\
				& \le \frac{\varepsilon \omega_n}{1+\varepsilon }\int_{\rn}A\left(\frac{1+\varepsilon}{\varepsilon}
				\left(\frac{3}{2}\right)^{s}\frac{|u(y)|}{|y|^{s}}\right)\;dy\,. \nonumber
\end{align}
Note that the last inequality holds since the set $\{|x-y|>2|x|\}$ is agrees with
the ball centered at $-\frac{1}{3}y$, with  radius   $\frac{2}{3}|y|$.
\\
In order to estimate the integral $I_1$,
observe that
$$\int_{\{|x-y|>2|x|\}}A\left((1+\varepsilon)\frac{
				|u(x)-u(y)|}{|x-y|^s}\right)\frac{dx\,dy}{|x-y|^{n}} = \int_{\{|x-y|>2|y|\}}A\left((1+\varepsilon)\frac{|u(x)-u(y)|}{|x-y|^s}\right)\frac{dx\,dy}{|x-y|^{n}}\,.$$
Furthermore, $\{|x-y|>2|x|\} \cap \{|x-y|>2|y| \}= \emptyset$, since   if there existed $x,y$ such that $|x-y|>2|x|$ and $|x-y|>2|y|$, then
$|x-y| \le |x|+|y| < \frac{|x-y|}{2}+\frac{|x-y|}{2}=|x-y|$, a contradiction. Thus,
\begin{align}\label{E:I-1}
		I_1
 \le \frac{1}{2(1+\varepsilon)}\int_{\rn} \int_{\rn}A\left((1+\varepsilon)\frac{
				|u(x)-u(y)|}{|x-y|^s}\right)\frac{dx\,dy}{|x-y|^{n}}.
\end{align}
Consequently,
\begin{align}\label{E:page-10}
		\frac{s}{1+\varepsilon} & \int_{\rn} \int_{\rn}A\left((1+\varepsilon)\frac{
		|u(x)-u(y)|}{|x-y|^s}\right)\frac{dx\,dy}{|x-y|^{n}}
		\ge 2sI_1
		 \ge  \frac{2\,\omega_n}{n} \int_{\rn}\overline{A}\left(\frac{|u(x)|}{2^s|x|^s}\right)\;dx - 2sI_2
			\\ \nonumber
		& \ge \frac{2\,\omega_n}{n} \int_{\rn}\overline{A}\left(\frac{|u(x)|}{2^s|x|^s}\right)\;dx
		 - \frac{2s\varepsilon \omega_n}{1+\varepsilon }\int_{\rn} A\left(\frac{1+\varepsilon}{\varepsilon}	\left(\frac{3}{2}\right)^{s}\frac{|u(y)|}{|y|^s}\right)\;dy\,,
\end{align}
where the first inequality follows from \eqref{E:I-1}, the second one is due to \eqref{E:3} and \eqref{E:2'}, and the third one to \eqref{E:4}. Since $A(t)\le \overline{A}(2t)$  for $t\ge0$,
 inequality \eqref{E:31} implies that
\begin{equation}\label{E:33}
	A(\lambda t)\le C\lambda^{I(A)+1}\,\overline{A}(2t)\qquad\text{if $t\ge0$ and $\lambda\ge1$\,.}
\end{equation}
From inequalities \eqref{E:page-10} and \eqref{E:33} one deduces that
\begin{align}\label{jan7}
    & \frac{s}{1+\varepsilon}\int_{\R^n}\int_{\R^n}A\left((1+\varepsilon)\frac{|u(x)-u(y)|}{|x-y|^s}\right)\frac{dx\,dy}{|x-y|^{n}}
        	\\ \nonumber
    & \ge \frac{2\, \omega_n}{n}\int_{\R^n}\overline{A}\left(\frac{|u(x)|}{2^s|x|^{s}}\right)\;dx
		- 	\frac{2Cs\varepsilon \omega_n}{1+\varepsilon}	\left(\frac{2(1+\varepsilon)}{\varepsilon}3^s\right)^{I(A)+1}
	\int_{\R^n}\overline A\left(\frac{|u(y)|}{2^s|y|^s}\right)\;dy
 			\\ \nonumber
    & =  \frac{2 \,\omega_n}{n} \left[1 -    
    \frac{Cns \varepsilon }{1+\varepsilon}
     \left(\frac{2(1+\varepsilon)}{\varepsilon}3^s\right)^{I(A)+1}
		\right]\int_{\R^n}\overline{A}\left(\frac{|u(x)|}{2^s|x|^{s}}\right)\;dx\,.
\end{align}
Thus, there exists $s _2 = s_2(A, n, \varepsilon) \in (0, s _1)$ such that, if $s \in (0, s_2)$, then
\begin{equation}\label{jan8}
\frac{s}{1+\varepsilon}\int_{\R^n}\int_{\R^n}A\left((1+\varepsilon)\frac{|u(x)-u(y)|}{|x-y|^s}\right)\frac{dx\,dy}{|x-y|^{n}}
 \geq  \frac{2\, \omega_n}{n}(1-\varepsilon)\int_{\R^n}\overline{A}\left(\frac{|u(x)|}{2^s|x|^{s}}\right)\;dx\,.
\end{equation}
On replacing $u$ by $u/(1+\varepsilon)$ in inequality \eqref{jan8}, one can infer, via Fatou's lemma, that
\begin{equation}\label{jan9}
	\liminf_{s\to0^+} s\int_{\R^n}
    		\int_{\rn} A\left(\frac{|u(x)-u(y)|}{|x-y|^s}\right)\frac{dx\,dy}{|x-y|^{n}}
    			\geq
    			\frac{2\, \omega_n}{n}(1-\varepsilon^{2})\int_{\R^n}\overline{A}\left(\frac{|u(x)|}{1+\varepsilon}\right)\;dx\,.
\end{equation}
By the arbitrariness of $\varepsilon$,
\begin{align}\label{E:6}
		\liminf_{s\to0^+} s & \int_{\rn} \int_{\rn}A\left(\frac{
		|u(x)-u(y)|}{|x-y|^s}\right)\frac{dx\,dy}{|x-y|^{n}}
		\ge \frac{2\, \omega_n}{n}\int_{\rn}\overline{A}\left(|u(x)|\right)\; dx\,.
\end{align}
In particular, inequality \eqref{E:6} implies that, if the integral on the right-hand side diverges, then equation \eqref{jan1} certainly holds. Thus, in what follows, we may  assume  that
\begin{equation}\label{jan100}
\int_{\rn}\overline{A}\left(|u(x)|\right)\; dx < \infty\, .
\end{equation}
Next, we provide an upper   bound for the $\limsup _{s\to 0^+}$ of the expression on the left-hand side of equation \eqref{jan1}. One has that
\begin{align}\label{E:8}
		& s  \int_{\R^n}  \int_{\rn} A\left(\frac{|u(x)-u(y)|}{|x-y|^s}\right)\frac{dx\,dy}{|x-y|^{n}}
				\\ \nonumber
		& =	s\int_{\R^n} \int_{\{|y|\ge|x|\}} A\left(\frac{|u(x)-u(y)|}{|x-y|^s}\right)\frac{dx\,dy}{|x-y|^{n}}
		 + s\int_{\R^n} \int_{\{|y|<|x|\}} A\left(\frac{|u(x)-u(y)|}{|x-y|^s}\right)\frac{dx\,dy}{|x-y|^{n}}
				\\ \nonumber
		& =	2s\int_{\R^n} \int_{\{|y|\ge|x|\}} A\left(\frac{|u(x)-u(y)|}{|x-y|^s}\right)\frac{dx\,dy}{|x-y|^{n}}
				\\ \nonumber
		& =	2s  \int_{\R^n} \int_{\{|y|\geq 2|x|\}} A\left(\frac{|u(x)-u(y)|}{|x-y|^s}\right)\frac{dx\,dy}{|x-y|^{n}} + 2s \int_{\R^n} \int_{\{|x|\le|y|< 2|x|\}} A\left(\frac{|u(x)-u(y)|}{|x-y|^s}\right)\frac{dx\,dy}{|x-y|^{n}} 
				\\ \nonumber
		& \le  \frac{2s}{1+\varepsilon}\int_{\R^n} \int_{\{|y|\ge2|x|\}} A\left((1+\varepsilon)\frac{|u(x)|}{|x-y|^s}\right)\frac{dx\,dy}{|x-y|^{n}}
				 +  \frac{2s\, \varepsilon}{1+\varepsilon}\int_{\R^n} \int_{\{|y|\ge2|x|\}} A\left(\frac{1+\varepsilon}{\varepsilon}
			\frac{|u(y)|}{|x-y|^s}\right)\frac{dx\,dy}{|x-y|^{n}}
				\\ \nonumber
		&	\quad + 2s\int_{\R^n} \int_{\{|x|\le|y|<2|x|\}} A\left(\frac{|u(x)-u(y)|}{|x-y|^s}\right)\frac{dx\,dy}{|x-y|^{n}}
				\\ \nonumber
		&  = J_1+J_2+J_3\,,
\end{align}
where the  inequality holds since $A$ is convex.
Let us  estimate  $J_1$ first.
To this purpose, notice that $\{|y|\ge2|x|\} \subset \{|x-y|\ge |x|\}$, since $|x-y| \ge |y|-|x| \ge 2|x|-|x| = |x|$.
Thus,
\begin{align}\label{E:page-2}
	\int_{\{|y|\ge2|x|\}} A\left((1+\varepsilon)\frac{|u(x)|}{|x-y|^s}\right)\frac{dy}{|x-y|^{n}}
		& \le \int_{\{|x-y|\ge |x|\}} A\left((1+\varepsilon)\frac{|u(x)|}{|x-y|^s}\right)\frac{dy}{|x-y|^{n}}
			\\ \nonumber
		& = \frac{\omega_n}{n}\int_{|x|}^{\infty}A\left((1+\varepsilon)\frac{|u(x)|}{r^s}\right)\frac{dr}{r}\,
\end{align}
for every  $x\in\R^{n}$.
A change of variables tells us that
\begin{equation}\label{feb40}
	\int_{t}^{\infty}A\left(\frac{(1+\varepsilon)\varrho}{r^s}\right)\frac{dr}{r}
		= \frac{1}{s}\int_{0}^{\frac{(1+\varepsilon)\varrho}{t^s}}\frac{A(\tau)}{\tau}\,d\tau
		= \frac{1}{s}\overline{A}\left(\frac{(1+\varepsilon)\varrho}{t^s}\right)\qquad \text{for $t, \varrho \ge0$\,.}
\end{equation}
Thanks to equations \eqref{E:page-2} and \eqref{feb40},
\begin{equation}\label{feb12}
	J_1 \le \frac{2\,\omega_n}{n(1+\varepsilon)}\int_{\R^{n}}\overline{A}\left((1+\varepsilon)\frac{|u(x)|}{|x|^s}\right)\;dx\,.
\end{equation}
As far as the term $J_2$ is concerned, observe that, if $|y|\ge2|x|$, then $|x-y|\ge \frac 12|y|$.
Therefore, an application of Fubini's theorem tells us that
%
%
\begin{align}\label{E:J-2}
		J_2 & 
			\le \frac{2^{n+1}s \varepsilon  }{1+\varepsilon}\int_{\R^n}A\left(\frac{1+\varepsilon}{\varepsilon} 2^{s}
	\frac{|u(y)|}{|y|^s}\right)\left(\int_{\{|x|\leq\frac{|y|}{2}\}}\;dx\right)\frac{dy}{|y|^n}
 = \frac{2\omega_n s \varepsilon }{1+\varepsilon}\int_{\R^n}A\left(\frac{1+\varepsilon}{\varepsilon} 2^{s}
			\frac{|u(y)|}{|y|^s}\right)\;dy.
\end{align}
 In order  to provide an upper bound for $J_3$, note that, given $r>3$,
\begin{align}\label{E:J-3}
		J_3 & 
			= 2s\int_{\R^n} \int_{\{|x|\le|y|<2|x|,|x-y|<r\}}A\left(
			\frac{|u(x)-u(y)|}{|x-y|^s}\right)\frac{dx\,dy}{|x-y|^{n}}
				\\ \nonumber
			& \quad + 2s\int_{\R^n} \int_{\{|x|\le|y|<2|x|,|x-y|\ge r\}}A\left(
			\frac{|u(x)-u(y)|}{|x-y|^s}\right)\frac{dx\,dy}{|x-y|^{n}}
				\\ \nonumber
			& = J_{31}+J_{32}.
\end{align}
Since we are assuming that $u\in\bigcup_{s\in(0,1)}V_d^{s,A}(\R^{n})$, there exists    $s_3\in(0,1)$
such that $u\in V_d^{s_3,A}(\R^{n})$.  Let $s \in (0,  s_3)$. Then
\begin{equation}\label{feb11}
	J_{31} \le 2s\int_{\R^n} \int_{\{|x|\le|y|<2|x|,|x-y|<r\}}A\left(
			\frac{|u(x)-u(y)|}{|x-y|^{ s_3}}r^{s_3-s}\right)\frac{dx\,dy}{|x-y|^{n}},
\end{equation}
since $s_3-s>0$ and $u\in V^{s_3,A}_d(\R^{n})$. If $|x|\le|y|<2|x|$ and $|x-y|\ge r$, then
\begin{equation*}
	3|x|=2|x|+|x|\ge|y|+|x|\ge|x-y|\ge r,
\end{equation*}
whence $|x|\ge \frac r3$, and $|x|\le|y|<|x|\le2|y|$.  Consequently,
\begin{equation*}
	3|y|=2|y|+|y|\ge2|x|+|y| \ge|x-y|\ge r,
\end{equation*}
and hence $|y|\ge\frac r3$ as well. Therefore, owing to the convexity of the function $A$,
\begin{align}\label{E:J-32}
		J_{32}&
		\le s\int_{\R^n} \int_{\{|x|\le|y|<2|x|,|x-y|\ge r\}}A\left(
			\frac{2|u(x)|}{|x-y|^s}\right)\frac{dx\,dy}{|x-y|^{n}}
				 + s\int_{\R^n} \int_{\{|x|\le|y|<2|x|,|x-y|\ge r\}}A\left(
			\frac{2|u(y)|}{|x-y|^s}\right)\frac{dx\,dy}{|x-y|^{n}}
				\\ \nonumber		
			& \le s\int_{\{|x|\geq\frac{r}{3}\}} \left(\int_{\{|x-y|\ge r\}}A\left(
			\frac{2|u(x)|}{|x-y|^s}\right)\frac{dy}{|x-y|^{n}}\right)\;dx
				 + s\int_{\{|y|\geq\frac{r}{3}\}} \left(\int_{\{|x-y|\ge r\}}A\left(
			\frac{2|u(y)|}{|x-y|^s}\right)\frac{dx}{|x-y|^{n}}\right)\;dy
				\\ \nonumber
			& = 2s\int_{\{|x|\geq\frac{r}{3}\}} \left(\int_{\{|x-y|\ge r\}}A\left(
			\frac{2|u(x)|}{|x-y|^s}\right)\frac{dy}{|x-y|^{n}}\right)\;dx
			= \frac{2s\omega_n}{n}\int_{\{|x|\geq\frac{r}{3}\}} \left(\int_{r}^{\infty}A\left(
			\frac{2|u(x)|}{\varrho ^s}\right)\frac{d \varrho}{\varrho}\right)\;dx
				\\ \nonumber	
			& = \frac{2\omega_n}{n}\int_{\{|x|\geq\frac{r}{3}\}} \overline{A}\left(
		\frac{2|u(x)|}{|x|^s}\right)\;dx.
\end{align}
Since we are assuming that $r>3$, the latter equation implies that
\begin{equation*}
	J_{32} \le \frac{2 \omega_n}{n}\int_{\{|x|\geq\frac{r}{3}\}} \overline{A}\left(
			2|u(x)|\right)\;dx \qquad\text{for every $s\in(0,1)$.}
\end{equation*}
Consequently, if $r$ is large enough,  then
\begin{equation}\label{feb10}
	J_{32}<\varepsilon
\end{equation}
for every $s\in(0,1)$.
Combining equations \eqref{E:8},  \eqref{feb12}-\eqref{feb11} and \eqref{feb10} implies that, for every $s\in (0,1)$,
\begin{align}\label{E:88}
	s\int_{\R^n} & \int_{\rn} A\left(\frac{|u(x)-u(y)|}{|x-y|^s}\right)\frac{dx\,dy}{|x-y|^{n}}
			\\ \nonumber
		& \le \frac{2\omega_n}{n(1+\varepsilon)}\int_{\R^{n}}\overline{A}\left((1+\varepsilon)\frac{|u(x)|}{|x|^s}\right)\;dx
		+\frac{2\omega_n s \varepsilon} {1+\varepsilon}\int_{\R^n}A\left(\frac{1+\varepsilon}{\varepsilon} 2^{s}
		\frac{|u(y)|}{|y|^s}\right)\;dy
			\\ \nonumber
		& \quad + 2s\int_{\R^n} \int_{\{|x|\le|y|<2|x|,|x-y|<r\}}A\left(
			\frac{|u(x)-u(y)|}{|x-y|^{s_3}}r^{s_3-s}\right)\frac{dx\,dy}{|x-y|^{n}}
			  + \varepsilon.	\end{align}
Passage to the limit as $s\to 0^+$ in inequality  \eqref{E:88} can be performed as follows.
If $|y|\le 2$, then the function $(0,1) \ni s\mapsto A\left(\frac{1+\varepsilon}{\varepsilon} 2^{s}\frac{|u(y)|}{|y|^s}\right)$ is non-decreasing.  Thus, 
\begin{equation*}
	A\left(\frac{1+\varepsilon}{\varepsilon} 2^{s}\frac{|u(y)|}{|y|^s}\right)
		\le A\left(\frac{1+\varepsilon}{\varepsilon} 2^{s_3}\frac{|u(y)|}{|y|^{s_3}}\right)
		\qquad\text{for every $s\in(0,{s_3})$,}
\end{equation*}
and, since we are assuming that $u\in V^{s_3,A}_d(\rn)$, we have that
$$\int_{\rn}
A\left(\frac{1+\varepsilon}{\varepsilon} 2^{s_3}\frac{|u(y)|}{|y|^{s_3}}\right)\; dy< \infty,
$$
owing to \eqref{jan6}.
Inasmuch as
\begin{equation*}
	\lim_{s\to0^+} A\left(\frac{1+\varepsilon}{\varepsilon} 2^{s}\frac{|u(y)|}{|y|^s}\right)
		= A\left(\frac{1+\varepsilon}{\varepsilon} |u(y)|\right) \qquad \text{for $y\neq 0$,}
\end{equation*}
%
the dominated convergence theorem ensures that
\begin{equation}\label{jan101}
	\lim_{s\to0^+} \int_{\{|y|\le 2\}}A\left(\frac{1+\varepsilon}{\varepsilon} 2^{s}\frac{|u(y)|}{|y|^s}\right)\;dy
		= \int_{\{|y|\le 2\}} A\left(\frac{1+\varepsilon}{\varepsilon} |u(y)|\right)<\infty.
\end{equation}
On the other hand, if $|y|> 2$, then the function $(0,1) \ni s\mapsto A\left(\frac{1+\varepsilon}{\varepsilon} 2^{s}\frac{|u(y)|}{|y|^s}\right)$ is non-increasing.
Consequently, the monotone convergence theorem yields
\begin{equation}\label{jan102}
	\lim_{s\to0^+} \int_{\{|y|> 2\}}A\left(\frac{1+\varepsilon}{\varepsilon} 2^{s}\frac{|u(y)|}{|y|^s}\right)\;dy
		= \int_{\{|y|> 2\}}A\left(\frac{1+\varepsilon}{\varepsilon}  |u(y)|\right)<\infty.
\end{equation}
Equations \eqref{jan101} and \eqref{jan102} imply that
\begin{equation}\label{E:second-term}
	\lim_{s\to0^+}\frac{2\omega_n s\varepsilon }{1+\varepsilon}\int_{\rn}A\left(\frac{1+\varepsilon}{\varepsilon} 2^{s}
		\frac{|u(y)|}{|y|^s}\right)\;dy =0.
\end{equation}
An argument analogous to that of the proofs of equations \eqref{jan101} and \eqref{jan102} yields
\begin{equation}\label{feb41}
    \lim_{s\to0^+} \int_{\R^{n}}\overline{A}\left((1+\varepsilon)\frac{|u(x)|}{|x|^s}\right)\;dx
    =
    \int_{\rn} \overline{A}\left(
		(1+\varepsilon)|u(x)|\right)\;dx.
\end{equation}
Next, for every $s\in(0,{s_3})$,
\begin{align*}
	\int_{\R^n} &\int_{\{|x|\le|y|<2|x|,|x-y|<r\}}A\left(\frac{|u(x)-u(y)|}{|x-y|^{s_3}}r^{s_3-s}\right)\frac{dx\,dy}{|x-y|^{n}}
		\\ &\le \int_{\R^n} \int_{\{|x|\le|y|<2|x|,|x-y|<r\}}A\left(\frac{|u(x)-u(y)|}{|x-y|^{s_3}}r^{s_3}\right)\frac{dx\,dy}{|x-y|^{n}}< \infty.
\end{align*}
Observe that the convergence of the last integral is due to the fact that $u\in V^{s_3,A}_d(\rn)$ and $A$ satisfies the $\Delta_2$-condition.
Therefore,
\begin{equation}\label{E:third-term}
	\lim_{s\to0^+}2s\int_{\R^n} \int_{\{|x|\le|y|<2|x|,|x-y|<r\}}A\left(
			\frac{|u(x)-u(y)|}{|x-y|^{s_3}}r^{{s_3}-s}\right)\frac{dx\,dy}{|x-y|^{n}} =0.
\end{equation}
Thanks to equations ~\eqref{E:88},~\eqref{E:second-term}, \eqref{feb41} and~\eqref{E:third-term},
%
%
\begin{equation*}
	\limsup_{s\to0^+}\int_{\R^n}
	\int_{\rn} A\left(\frac{|u(x)-u(y)|}{|x-y|^s}\right)\frac{dxdy}{|x-y|^{n}}
	\le \frac{2\omega_n}{n(1+\varepsilon)}\int_{\rn} \overline{A}\left(
			(1+\varepsilon)|u(x)|\right)\,dx + \varepsilon.
\end{equation*}
Hence, owing to the arbitrariness of $\varepsilon$,
\begin{equation}\label{jan10}
	\limsup_{s\to0^+}\int_{\R^n}
	\int_{\rn} A\left(\frac{|u(x)-u(y)|}{|x-y|^s}\right)\frac{dxdy}{|x-y|^{n}}
	\le \frac{2\omega_n}{n}\int_{\rn} \overline{A}\left(
			|u(x)|\right)\,dx.
\end{equation}
Coupling equations  \eqref{E:6} and \eqref{jan10} yields \eqref{jan1}.
\end{proof}

\section{Proof of Theorem \ref{counterex}}

Functions $A$ and $u$ as in the statement of Theorem \ref{counterex} are  explicitly exhibited in our proof.
\begin{proof}[Proof of Theorem \ref{counterex}]
Let $\gamma>1$ and let $A$ be any finite-valued Young function such that
\begin{equation*}
A(t)= e^{-\frac1{t^{\gamma}}} \quad \hbox{for $t \in (0, \tfrac 1{2e})$\,.}
\end{equation*}
Note that functions $A$ enjoying this property do exist,  since
    $\lim_{t\to0^+}e^{-\frac1{t^{\gamma}}}=0$
and the function $e^{-\frac1{t^{\gamma}}}$ is convex on the interval $\big(0,\big(\frac{\gamma}{\gamma+1}\big)^{\frac{1}{\gamma}}\big)$. The fact that $A$ is a Young function ensures that, for every $t_0 >0$,
\begin{equation}\label{feb20}
A(t) \leq t \tfrac{A(t_0)}{t_0} \quad \text{for $t \in [0, t_0]$.}
\end{equation}
Also, one can verify that, for each $s \in (0,1)$, there exists $\overline{t}=\overline{t}(s,n) \in (0, \tfrac 1{2e})$ such that the function
\begin{equation}\label{feb22}
    (0, \overline t) \ni t \mapsto   \frac{A(t^{1-s})}{t^{\gamma}}\quad \text{is  increasing.}
\end{equation}
Let $v\colon\R^n\to \R$ be the function defined as
\begin{equation}\label{A7}
    v(x) =
    \begin{cases}
     \displaystyle\frac {x_1}{|x| \log^{\frac{1}{\gamma}} {(\kappa+|x|)} } &\text{if $|x|\geq 1$}
                    \\
              \\ \displaystyle \frac {x_1}{\log^{\frac{1}{\gamma}} {(\kappa+1)} } &\text{if $|x| < 1$,}
    \end{cases}
\end{equation}
where $x=(x_1, \dots , x_n)$ and $\kappa>1$  is a sufficiently large constant  to be chosen later in such a way the argument of the function $A$, evaluated at several expressions depending on $v$, belongs to the interval $(0, \tfrac 1{2e})$ or $(0, \overline t)$.
\\ Notice that the function $v$ is Lipschitz continuous in $\rn$ and continuously differentiable in $\{|x|>1\}$,
and
\begin{equation}\label{A50}
|\nabla v (x)| \le \frac{\kappa}{|x|  \log^{\frac{1}{\gamma}}{(\kappa+ |x|)}} \qquad \hbox{if}\; |x|>1\,.
\end{equation}
Moreover, if $x,y \in \rn$ are such that $|(1-\tau)x+\tau y|>1$ for $\tau\in[0,1]$, then there exists $\tau_0 \in [0,1]$ satisfying
\begin{equation}\label{A51}
|v(x)-v(y)| \le \frac{3|x-y|}{|(1-\tau_0)x+\tau_0 y|\log^{\frac{1}{\gamma}}{(\kappa + |(1-\tau_0)x+\tau_0 y|)} }\,.
\end{equation}
Given $\lambda>1$, choose $\kappa$ so large that $\frac 1 {\lambda \log^{\frac{1}{\gamma}}{(\kappa+1)}}<\frac 1{2e}$. Therefore, there exists a constant $C$ such that 
\begin{align}\label{A52}
&\int_{\rn} A\left( \frac{|v(x)|}{\lambda}\right)\; dx 
\leq
C
+ \int_{|x|\ge1} \frac {dx}{ (\kappa+|x|)^{\lambda^{\gamma}}} <\infty\,.\nonumber
\end{align}
Now, we claim that
\begin{equation}\label{A53}
\int_{\R^n} \int_{\R^n} A\left(\frac{|v(x) - v(y)|}{\lambda |x-y|^s}\right) \; \frac{dx\,dy}{|x-y|^{n}} <  \infty
\end{equation}
for every $s\in (0,1)$ and $\lambda \geq 1$. Fix any $s\in (0,1)$.
To verify equation \eqref{A53}, observe that
\begin{align}\label{A54}
&\int_{\R^n} \int_{\R^n} A\left(\frac{|v(x) - v(y)|}{|x-y|^s}\right) \; \frac{dx\,dy}{|x-y|^{n}} = \int \int_{\{|x|\leq 1,|y|\leq 1\}}A\left(\frac{|v(x) - v(y)|}{|x-y|^s}\right) \; \frac{dx\,dy}{|x-y|^{n}}
\\
& \quad + 2
\int \int_{\{|x|<1,|y|> 1\}}A\left(\frac{|v(x) - v(y)|}{|x-y|^s}\right) \; \frac{dx\,dy}{|x-y|^{n}} + \int \int_{\{|x|> 1,|y|> 1\}}A\left(\frac{|v(x) - v(y)|}{|x-y|^s}\right) \; \frac{dx\,dy}{|x-y|^{n}}\nonumber
\\
& =J_1+ J_2 + J_3\,. \nonumber
\end{align}
Owing to the Lipschitz continuity of $v$ and to property \eqref{feb20}, if $E\subset \R^{n}\times\R^{n}$ is a bounded set, then there exist positive constants $C$ and $C'$  such that
\begin{align}\label{E:12}
&\int_{E} \int_{E} A\left(\frac{|v(x) - v(y)|}{|x-y|^s}\right) \; \frac{dx\,dy}{|x-y|^{n}} \le \int_{E} \int_{E}A\left(C|x-y|^{1-s}\right) \; \frac{dx\,dy}{|x-y|^{n}}
\le C'
\int_{E} \int_{E}\frac{dx\,dy}{|x-y|^{n-1+s}}<\infty.
\end{align}
Hence,
\begin{equation}\label{E:13}
J_1<\infty.
\end{equation}
Next, let us split  $J_2$ as
\begin{align}\label{A14}
J_2 &=\int\int_{\{|x|\leq 1,|y|>1, |x-y|>2\}}
A \left( \frac{|v(x) - v(y)|}{|x-y|^s}\right)\; \frac{dx\, dy}{|x-y|^n}
\\
& \quad + \int\int_{\{|x|\leq 1, |y|>1, |x-y|\leq 2\}}
A \left( \frac{|v(x) - v(y)|}{|x-y|^s}\right)\; \frac{dx\, dy}{|x-y|^n} = J_{21} + J_{22}\nonumber
\end{align}
Consider $J_{21}$ first. If $|x|\leq 1$ and $|x-y|>2$, then
$$|x|+|y| \leq |x| + |y-x| +|x| = 2|x| + |y-x| \leq 2 + |y-x| \leq |x-y|+|y-x| =2|x-y|\,.$$
Thus,
   $$ |x-y|\leq |x|+|y|\leq 2|x-y|\,.$$
Hence, there exist  positive constants $C,C',C''$ such that
\begin{align}\label{A17}
J_{21}&\leq \int\int_{\{|x|\leq 1,|y|>1, |x-y|>2\}}
A \left( C \frac{|v(x)|+ |v(y)|}{(|x| +|y|)^s}\right)\; \frac{dx\, dy}{(|x|+|y|)^{n}}
\\
&\leq \int\int_{\{|x|\leq 1,|y|>1\}}
A \left( 2 C \frac{|v(x)|}{ |y|^s}\right)\; \frac{dx\, dy}{|y|^{n}}+ \int\int_{\{|x|\leq 1,|y|>1\}}
A \left( 2 C \frac{|v(y)|}{ |y|^s}\right)\; \frac{dx\, dy}{|y|^{n}}\nonumber
\\
&\leq 2 \int\int_{\{|x|\leq 1,|y|>1\}} A\left( \frac {2C}{\log^{\frac{1}{\gamma}}( \kappa +1 )\,|y|^s}\right)\; \frac{dx\, dy}{|y|^{n}}
= C' \int_1^\infty A\left( \frac{2C}{\log^{\frac{1}{\gamma}}( \kappa +1 )\,r^s}  \right) \; \frac{dr}{r} \nonumber \\
&= C'\int_1^\infty e^{-C''r^{\gamma s}}  \; \frac{dr}{r}
<\infty, \nonumber
\end{align}
where
the last equality holds provided that the constant $\kappa$ is so large that $\frac{2C}{\log^{\frac{1}{\gamma}}( \kappa+1 )} < \frac 1{2e}$.
\\ As for  $J_{22}$, notice that, if $|x|\leq 1$ and $|x-y|\leq 2$, then
$|y|\leq |y-x| +|x| \leq 3$.
 Thus, by property~\eqref{E:12}, one has that
 \begin{align}\label{A19}
J_{22} < \infty.
\end{align}
Finally, let us focus on the term $J_3$, that can be split as
\begin{align}\label{A20}
J_3 &= \int\int_{\{|x|>1, |y|>1, |x-y|\geq \frac{ |x|+|y|}{2}\}}  A\left( \frac{|v(x) - v(y)|}{|x-y|^s} \right) \; \frac{dx\, dy}{|x-y|^n}
\\
& \quad
+ \int\int_{\{|x|>1, |y|>1, |x-y|< \frac{ |x|+|y|}{2}\}}  A\left( \frac{|v(x) - v(y)|}{|x-y|^s} \right) \; \frac{dx\, dy}{|x-y|^n}
= J_{31}+ J_{32}\,.\nonumber
\end{align}
Consider $J_{32}$. If
\begin{equation}
    \label{feb30}
    |x-y|< \frac{|x|+|y|}{2},
\end{equation}
 then
$|x|\leq |x-y| +|y| \leq \frac {|x|}2 + \frac{|y|}2 +|y|\,,$
whence $|x|\leq 3|y|$. Similarly, $|y|\leq 3|x|$. Thus,
$\frac {|y|}3 \leq |x| \leq 3|y|$,
and
\begin{equation}\label{A22}
|x|  \geq \frac{|x| +|y|}6, \quad |y|\ge\frac{|x| +|y|}6.
\end{equation}
Moreover, if $x$ and $y$ fulfill inequality \eqref{feb30},  then there exists an absolute constant $\beta>0$ such that
\begin{equation}\label{E:19}
    |(1-\tau)x+\tau y|\ge\beta(|x|+|y|)\quad \text{for $\tau\in[0,1]$.}
\end{equation}
Indeed, squaring  both sides of inequality \eqref{feb30}  shows that it is equivalent to
\begin{align*}
  8x\cdot y  > 
  2(|x|^2+|y|^2)+(|x|-|y|)^2.
\end{align*}
Hence, $x\cdot y  > \frac 14 (|x|^2+|y|^2)$ and, by inequality \eqref{A22}, there exists an absolute constant $C$ such that
\begin{align}\label{E:20}
   |(1-\tau)x+\tau y|^{2}
   & = (1-\tau)^{2}|x|^{2}+2\tau(1-\tau)x\cdot y+\tau^2|y|^2
        \\ \nonumber
   & \ge (1-\tau)^{2}|x|^{2}+\tau(1-\tau)\frac{|x|^2+|y|^2}{2}+\tau^2|y|^2
        \\ \nonumber
   & \ge C\min\left\{|x|^2,|y|^2\right\} \ge C \left(\frac{|x|+|y|}{6}\right)^{2}  \quad \text{for $\tau \in[0,1]$.}
\end{align}
Inequality~\eqref{E:19} is thus established.
Let us split $J_{32}$ as
\begin{align}\label{E:22}
  J_{32} & = \int\int_{\{|x|>1, |y|>1, |x-y|< \frac{ |x|+|y|}{2}, \sqrt{|x|^2+|y|^2}<\frac{1}{\beta}\}} A\left( \frac{|v(x) - v(y)|}{|x-y|^s} \right) \; \frac{dx\, dy}{|x-y|^n}
        \\ \nonumber
    & \quad + \int\int_{\{|x|>1, |y|>1, |x-y|< \frac{ |x|+|y|}{2}, \sqrt{|x|^2+|y|^2}\ge\frac{1}{\beta}\}}A\left( \frac{|v(x) - v(y)|}{|x-y|^s} \right) \; \frac{dx\, dy}{|x-y|^n}
        \\ \nonumber &  = J_{321}+J_{322}.
\end{align}
By property~\eqref{E:12},
\begin{equation}\label{E:23}
  J_{321}<\infty.
\end{equation}
 As for $J_{322}$,
note that if $x,y$ are such that $\sqrt{|x|^2+|y|^2}\ge\frac{1}{\beta}$, then
\begin{equation}\label{E:21}
  1\leq\beta\sqrt{|x|^2+|y|^2}\le \beta(|x|+|y|).
\end{equation}
 If   $c$ is sufficiently large, then the following chain holds for a suitable constant $C$:
 \begin{align}\label{E:24}
    J_{322}
    & \le \int\int_{\{|x|>1, |y|>1, |x-y|< \frac{ |x|+|y|}{2}, \sqrt{|x|^2+|y|^2}\ge\frac{1}{\beta}\}} A\left(\frac{3|x-y|^{1-s}}
        {\log^{\frac{1}{\gamma}}\left(\kappa+\beta\left(|x|+|y|\right)\right)   \beta\left(|x|+|y|\right)} \right) \; \frac{dx\, dy} {|x-y|^n}
            \\ \nonumber
    & \le \int\int_{\{|x|>1, |y|>1, |x-y|< \frac{ |x|+|y|}{2}, \sqrt{|x|^2+|y|^2}\ge\frac{1}{\beta}\}} A\left(\frac{3}
        {\log^{\frac{1}{\gamma}}\left(\kappa+\beta\left(|x|+|y|\right)\right) \beta\left(|x|+|y|\right)^{s}} \right) \; \frac{dx\, dy} {\left(|x|+|y|\right)^n}
            \\ \nonumber
    & = \int\int_{\{|x|>1, |y|>1, |x-y|< \frac{ |x|+|y|}{2}, \sqrt{|x|^2+|y|^2}\ge\frac{1}{\beta}\}} e^{-C\log\left(\kappa+\beta\left(|x|+|y|\right)\right)\left(|x|+|y|\right)^{\gamma s}}
         \; \frac{dx\, dy} {\left(|x|+|y|\right)^n}
            \\ \nonumber
   & \le \int\int_{\{|x|>1, |y|>1\}} \kappa^{-C\left(|x|+|y|\right)^{\gamma s}}
         \; \frac{dx\, dy} {\left(|x|+|y|\right)^n} <\infty,
\end{align}
 where the first inequality holds owing to ~\eqref{E:19}, \eqref{E:21}, \eqref{A51}, and the second one by   property~\eqref{feb22}.
Equations~\eqref{E:22}--\eqref{E:24} ensure that
\begin{equation}\label{E:25}
    J_{32}<\infty.
\end{equation}
It remains to estimate $J_{31}$. The following chain holds, provided that $\kappa$ is sufficiently large:
\begin{align}\label{A56}
    J_{31}
        & \le \int\int_{\{|x|>1, |y|>1, |x-y|\geq \frac{ |x|+|y|}{2}\}}  A\left( \frac{2|v(x)|}{|x-y|^s} \right) \; \frac{dx\, dy}{|x-y|^n}
            \\ \nonumber
        & \quad + \int\int_{\{|x|>1, |y|>1, |x-y|\geq \frac{ |x|+|y|}{2}\}}  A\left( \frac{2|v(y)|}{|x-y|^s} \right) \; \frac{dx\, dy}{|x-y|^n}
            \\ \nonumber
        & = 2 \int\int_{\{|x|>1, |y|>1, |x-y|\geq \frac{ |x|+|y|}{2}\}}  A\left( \frac{2|v(x)|}{|x-y|^s} \right) \; \frac{dx\, dy}{|x-y|^n}
            \\ \nonumber
        & \le 2 \int\int_{\{|x|>1, |y|>1, |x-y|\geq \frac{ |x|+|y|}{2}\}}  A\left( \frac{2}{\log^{\frac{1}{\gamma}}\left(\kappa+1\right)|x-y|^s} \right) \; \frac{dx\, dy}{|x-y|^n}
            \\ \nonumber
        & \le 2^{n+1} \int\int_{\{|x|>1, |y|>1, |x-y|\geq \frac{ |x|+|y|}{2}\}}  A\left( \frac{2^{s+1}}{\log^{\frac{1}{\gamma}}\left(\kappa+1\right)\left(|x|+|y|\right)^s} \right) \; \frac{dx\, dy}{\left(|x|+|y|\right)^n}
            \\ \nonumber
        & \le 2^{n+1} \int\int_{\{|x|>1, |y|>1\}}  e^{-\frac{\log(\kappa+1)}{2^{(s+1)\gamma}}\left(|x|+|y|\right)^{\gamma s}} \; \frac{dx\, dy}{\left(|x|+|y|\right)^n}
            \\ \nonumber
        & = 2^{n+1} \int\int_{\{|x|>1, |y|>1\}}  (\kappa+1)^{-\left(\frac{\left(|x|+|y|\right)^{s}}{2^{(s+1)}}\right)^{\gamma}}  \; \frac{dx\, dy}{\left(|x|+|y|\right)^n}<\infty.
\end{align}
Property~\eqref{A53} follows from~\eqref{A54}, \eqref{E:13}, \eqref{A14}, \eqref{A17}, \eqref{A19}, \eqref{A20}, \eqref{E:25} and~\eqref{A56}.
\\ We conclude by proving that, if $\lambda \in (1,2)$, then
\begin{equation}\label{E:27bis}
  \lim_{s\to0^+}s\int_{\R^{n}} \int_{\R^{n}} A\left(\frac{|v(x)- v(y)|}{\lambda |x-y|^s}\right)\; \frac{dx\,dy}{|x-y|^{n}} = \infty.
\end{equation}
To this purpose, note that, given $\sigma\in(0,1)$,
\begin{align}\label{E:28bis}
    &\int_{\R^{n}}  \int_{\R^{n}} A\left(\frac{|v(x)- v(y)|}{\lambda |x-y|^s}\right)\; \frac{dx\,dy}{|x-y|^{n}}
            \\ \nonumber
        & \ge \int\int_{\{|x|>1,x_1>\sigma|x|, |y|>1, y_1<-\sigma|y|\}} A\left(\frac{|v(x)- v(y)|}{\lambda |x-y|^s}\right)\; \frac{dx\,dy}{|x-y|^{n}}
            \\ \nonumber
        & = \int\int_{\{|x|>1,x_1>\sigma|x|, |y|>1, y_1<-\sigma|y|\}} A\left(\frac{1}{\lambda}\left(\frac{x_1}{|x|\log^{\frac{1}{\gamma}}\left(\kappa+|x|\right)}
            - \frac{y_1}{|y| \log^{\frac{1}{\gamma}}\left(\kappa+|y|\right) }\right)\frac{1}{|x-y|^s}\right)\; \frac{dx\,dy}{|x-y|^{n}}
            \\ \nonumber
        & \ge \int\int_{\{|x|>1,x_1>\sigma|x|, |y|>1, y_1<-\sigma|y|\}}
         A\left(\frac{2\sigma}{\lambda\log^{\frac{1}{\gamma}}\left(\kappa+|x| +|y|\right)}\frac{1}{\left(|x|+|y|\right)^{s}}\right)\; \frac{dx\,dy}{ (|x|+|y|)^n}
            \\ \nonumber
        & = C_{\sigma,n}\int_{1}^{\infty}\int_{1}^{\infty}
         A\left(\frac{2\sigma}{\lambda \log^{\frac{1}{\gamma}}\left(\kappa+\varrho+r\right)}\frac{1}{\left(\varrho+r\right)^{s}}\right)\; \frac{\varrho^{n-1}r^{n-1}}{\left(\varrho+r\right)^{n}}d\varrho\,dr
\end{align}
for some positive constant $C_{\sigma,n}$ depending on $\sigma$ and $n$. Note that the last equality follows on making use of the polar coordinates in the integral with respect to $x$ and in the integral with respect to $y$, owing to the fact that the integrand is a radial function in $x$ and $y$, respectively, and that each of the sets $\{|x|>1,x_1>\sigma |x|\}$ and $\{|y|>1,y_1<-\sigma |y|\}$ is the intersection of the exterior of a ball centered at $0$ with a cone whose vertex is also $0$. Via the change of variables
  $\xi=\varrho+r, \eta=\varrho-r$,
we obtain that
\begin{align}\label{E:29bis}
     \int_{1}^{\infty}&\int_{1}^{\infty}
         A\left(\frac{2\sigma}{\lambda \log^{\frac{1}{\gamma}}\left(\kappa+\varrho+r\right)}\frac{1}{\left(\varrho+r\right)^{s}}\right)\; \frac{\varrho^{n-1}r^{n-1}}{\left(\varrho+r\right)^{n}}d\varrho\,dr
                \\ \nonumber
    & = \frac{1}{2} \int_{2}^{\infty}\int_{2-\xi}^{-2 +\xi}
       A\left(\frac{2\sigma}{\lambda \log^{\frac{1}{\gamma}}\left(\kappa+\xi\right) }\frac{1}{\xi^{s}}\right)\; \frac{\left(\xi^{2}-\eta^{2}\right)^{n-1}}{4^{n-1}\xi^{n}}d\eta\,d\xi.
\end{align}
Given $\alpha \in (0, 2)$, if $\xi>\frac{4}{2-\alpha}$   and $2-\xi \le \eta \le \xi-2$, then $\xi^2-\eta^2 \ge \xi^2-(\xi-2)^2 =  4 \xi-4>\alpha\xi$.
Thereby, on choosing $\kappa$  large enough,  one has that
\begin{align}\label{E:31bis}
    \int_{2}^{\infty}&\int_{2-\xi}^{-2 +\xi}
          A\left(\frac{2\sigma}{\lambda\log^{\frac{1}{\gamma}}\left(\kappa+\xi\right)}\frac{1}{\xi^{s}}\right)\; \frac{\left(\xi^{2}-\eta^{2}\right)^{n-1}}{\xi^{n}}d\eta\,d\xi
                 \\ \nonumber
    & \ge \int_{\frac{4}{2-\alpha}}^{\infty}\int_{2-\xi}^{-2+\xi}
          A\left(\frac{2\sigma}{\lambda\log^{\frac{1}{\gamma}}\left(\kappa+\xi\right)}\frac{1}{\xi^{s}}\right)\; \frac{\left(\alpha\xi\right)^{n-1}}{\xi^{n}}d\eta\,d\xi
               >\alpha^{n} \int_{\frac{4}{2-\alpha}}^{\infty}
         A\left(\frac{2\sigma}{\lambda\log^{\frac{1}{\gamma}}\left(\kappa+\xi\right)}\frac{1}{\xi^{s}}\right)\,d\xi
                \\ \nonumber
              & = \alpha^{n} \int_{\frac{4}{2-\alpha}}^{\infty} e^{-\left(\frac{\lambda}{2\sigma}\right)^{\gamma}\log(\kappa+\xi)\xi^{\gamma s}}\,d\xi
                  =  \alpha^{n} \int_{\frac{4}{2-\alpha}}^{\infty} \frac{d\xi}{(\kappa+\xi)^{\left(\frac{\lambda}{2\sigma}\right)^{\gamma}\xi^{\gamma s}}}
                 \\ \nonumber
    & = \frac{ \alpha^{n}}{s} \int_{\left(\frac{4}{2-\alpha}\right)^{s}}^{\infty} \frac{t^{\frac{1}{s}}}{(\kappa+t^{\frac{1}{s}})^{\left(\frac{\lambda}{2\sigma}\right)^{\gamma}t^{\gamma}}}\frac{dt}{t}.
\end{align}
Now, fix any $\sigma \in (\frac{\lambda}{2}, 1)$.
Then $\left(\frac{\lambda}{2\sigma}\right)^{\gamma}<1$. Also, $\left(\frac{\lambda}{2\sigma}\right)^{\gamma}t^{\gamma}<1$ if $t<\frac{2\sigma}{\lambda}$. Thus,
\begin{equation}\label{E:32bis}
  \frac{\chi_{((\frac{4}{2-\alpha})^{s},\frac{2\sigma}{\lambda})}(t)
  t^{\frac{1}{s}}}{(\kappa+t^{\frac{1}{s}})^{\left(\frac{\lambda}{2\sigma}\right)^{\gamma}t^{\gamma}}} \nearrow \infty
    \quad \text{as} \quad  s\searrow 0^+ \quad \text{for $t \in (1, \tfrac{2\sigma}{\lambda})$}.
\end{equation}
Equation \eqref{E:27bis} follows from~\eqref{E:28bis}, \eqref{E:29bis}, \eqref{E:31bis} and \eqref{E:32bis}, via the monotone convergence theorem for integrals.
Altogether, we have shown that the conclusions of the theorem hold with $u=\frac{v}{\lambda}$ for any $\lambda\in(1,2)$.
\end{proof}

\section*{Compliance with Ethical Standards}\label{conflicts}

\subsection*{Funding}

This research was partly funded by:

\begin{enumerate}
\item Research Project 201758MTR2  of the Italian Ministry of Education, University and
Research (MIUR) Prin 2017 ``Direct and inverse problems for partial differential equations: theoretical aspects and applications'';
\item GNAMPA of the Italian INdAM -- National Institute of High Mathematics
(grant number not available);
\item Grant P201-18-00580S of the Czech Science Foundation;
\item Deutsche Forschungsgemeinschaft (DFG, German Research Foundation) under Germany's Excellence Strategy - GZ 2047/1, Projekt-ID 390685813.
\end{enumerate}

\subsection*{Conflict of Interest}

The authors declare that they have no conflict of interest.


{\color{black}

\end{document}